\documentclass[12pt]{amsart}

\usepackage[english]{babel}

\usepackage[letterpaper,top=3cm,bottom=3cm,left=3cm,right=3cm,marginparwidth=1.75cm]{geometry}

\usepackage{amsmath}
\usepackage{graphicx}
\usepackage{tikz}
\usepackage{tikz-cd}
\usepackage{amsfonts}
\usepackage{amssymb}
\usepackage{amsthm}
\usepackage{algorithm}
\usepackage{algpseudocode}
\usepackage{caption}
\usepackage[colorlinks=true, allcolors=violet]{hyperref}
\usepackage{cleveref}

\newtheorem{Thm}{Theorem}[section]
\newtheorem{Cor}[Thm]{Corollary}
\newtheorem{Lm}[Thm]{Lemma}
\newtheorem{Prop}[Thm]{Proposition}

\newtheorem{LetterThm}{Theorem}

\theoremstyle{definition}
\newtheorem{Defn}[Thm]{Definition}

\newtheorem{Ex}[Thm]{Example}
\newtheorem{Rem}[Thm]{Remark} 

\author{Eoin Mackall}
\address{Department of Mathematics, University of California, Santa Cruz, CA 95064}
\email{emackall@ucsc.edu}

\author{Diego Yépez}
\title[Canonical forms on anisotropic conics]{An algorithmic reduction to canonical forms for vector bundles on anisotropic conics}
\thanks{Throughout the majority of this work, the second author was supported by the IDA Postdoctoral Fellowship}
\address{IDA Center for Communications Research - Princeton, Princeton, NJ 08540}
\email{d.yepez@idaccr.org}
\address{Department of Mathematics, University of Colorado Boulder, Boulder, CO 80309}
\email{diego.yepez@colorado.edu}

\begin{document}

\begin{abstract}
We describe a polynomial complexity algorithm for reducing transition matrices, for vector bundles glued along a clutching-type cover of a real anisotropic conic, to canonical block diagonal forms. This is a generalization, to the real anisotropic form, of the classification of vector bundles on the Riemann sphere by their canonical diagonal forms due to Grothendieck and Birkhoff.

To enable our algorithm, we provide an elementary algebraic proof for the result, due to Biswas-Nagaraj and Novakovic, of the decomposition of vector bundles on real anisotropic conics into sums of indecomposable vector bundles of rank at most 2. While our algorithm and our proof of this decomposition focus solely on the setting of a real anisotropic conic, our methods are immediately generalizable to anisotropic conics over arbitrary fields.
\end{abstract}

\date{April 19, 2026}
\subjclass[2020]{14H60; 14F06; 13C10; 15A21}
\keywords{Vector Bundles; Conics, Severi-Brauer Curves, Decomposition}

\maketitle

\tableofcontents

\section{Introduction}
\subsection*{Background} There are various equivalent statements that describe the decomposition of vector bundles on the complex projective line into direct sum of line bundles, the most well-known of which are due, independently, to Grothendieck \cite{G1957} and to Birkhoff \cite{B1909}.
Formally, if $\mathcal{V}$ is either a rank $n$ holomorphic or algebraic vector bundle on the Riemann sphere, then there is an isomorphism
$$\mathcal{V} \cong \bigoplus\limits_{i=1}^{n}\mathcal{O}_{\mathbb{P}_{\mathbb{C}}^{1}}(\xi_{i})$$
with $\xi_{i} \in \mathbb{Z}$ and moreover the $\xi_{i}$ are unique up to index reordering. 
Over a decade after Grothendieck's proof of this result using cohomology \cite{G1957}, Hazewinkel and Martin in \cite{MR662762} provided what they called an ``elementary" proof of the splitting of vector bundles on $\mathbb{P}_{k}^{1}$, for an arbitrary field $k$, by reducing the problem to an equivalent form as a factorization of matrices. Nowadays, the elementary proof of Hazewinkel and Martin would more accurately be described as an algorithm for reducing transition matrices to an equivalent diagonal canonical form.

A natural setting where these results generalize is gotten by replacing the projective line with any nonsingular conic in the projective plane. Any such conic is isomorphic to the projective line if and only if it contains a point rational over the ground field -- these varieties form the simplest (i.e.\ one dimensional) examples of Severi--Brauer varieties. Corresponding decompositions of vector bundles on a nonsingular plane conic into a sum of indecomposables have been shown to exist by Biswas and Nagaraj \cite{BN2005, BN2007, BN2009} and, independently, by Novakovi{\'c} \cite{Nov2012, Nov2024}. While we work primarily over the real anisotropic conic for concreteness, our methods apply more broadly: the algorithm and classification extend to anisotropic conics over any infinite field, with only minor modifications to the arguments.

Over conics, vector bundles decompose as follows: 

\begin{LetterThm}[Theorem \ref{thm: cantypes}]
Let $k$ be an arbitrary field and let $C\subset \mathbb{P}^2_k$ be a nonsingular conic, i.e.\ there is a regular quadratic form $q$ in three variables such that $C=V(q)$. Let $\Omega_{C}$ denote the cotangent bundle (sheaf of K{\"a}hler differentials) on $C$ and let $\mathcal{V}$ be an arbitrary vector bundle on $C$.

Then there exists a canonical indecomposable rank $2$ bundle $\mathcal{Q}$ on $C$ such that there is an isomorphism $$\mathcal{V} \cong \Big(\bigoplus\limits_{i = 1}^{r_{0}} \Omega_C^{\otimes m_{i}}\Big) \oplus \Big(\bigoplus\limits_{j = 1}^{r_{1}} \mathcal{Q} \otimes \Omega_C^{\otimes n_{j}}\Big)$$ for unique, up to reordering, integer sequences $\{m_i\}_{i=1}^{r_0}$ and $\{n_j\}_{j=1}^{r_1}$.
\end{LetterThm}

Our main contribution, in this article, is to provide an algorithm to explicitly realize such decompositions in the case of the real anisotropic conic described as the solution set of the equation $x^2+y^2+z^2=0$. Specifically, we prove:

\begin{LetterThm}
Let $C=V(x^2+y^2+z^2)\subset \mathbb{P}^2_\mathbb{R}$. Let $U_1=D_+(z)$ and $U_2=D_+(y)$ be the given basic opens, which cover $C$, and write \[R_1=\Gamma(U_1,\mathcal{O}_{U_1}),\quad R_2=\Gamma(U_2,\mathcal{O}_{U_2}), \quad \mbox{and} \quad R_{12}=\Gamma(U_1\cap U_2,\mathcal{O}_{U_1\cap U_2})\] for the associated coordinate rings. Then any vector bundle $\mathcal{V}$ on $C$ is describable by a canonical transition matrix $M(\mathcal{V})$ on $U_1\cap U_2$ as a block sum $M(\mathcal{V})=\bigoplus_{i\in \mathbb{Z}} M_i^{\oplus m_i}$, for uniquely determined integers $m_i\geq 0$, where \[M_i= (-y/z)^{i/2} \quad \mbox{or}\quad M_i=\begin{bmatrix} (-y/z)^{(i+1)/2}  & (x/z)(-y/z)^{(i-1)/2}  \\ 0  & (-y/z)^{(i-1)/2} \end{bmatrix}\] depending if $i\equiv 0 \pmod{2}$ or $i\equiv 1 \pmod{2}$ respectively.

Moreover, there is an efficient algorithm (i.e.\ an algorithm that is polynomial, in the number of field operations needed, in specific inputs), for reducing any transition matrix $M\in \mathrm{GL}_n(R_{12})$ describing $\mathcal{V}$ to a block diagonal matrix $M(\mathcal{V})$ of the above form by multipliers $L\in \mathrm{GL}_n(R_2)$ and $R\in \mathrm{GL}_n(R_1)$, i.e.\ satisfying \[LMR=M(\mathcal{V}).\] If the multipliers $L$ and $R$ are needed, then they can similarly be found with a polynomial algorithm in the same inputs.
\end{LetterThm}

Theorem B is a combination of several distinct results that span the entirety of this article (Theorems \ref{thm: cantypes}, \ref{thm: alg}, and \ref{thm: multipliers}). The proofs for these theorems rely almost entirely on elementary algebraic geometry and linear algebra but, we also occasionally appeal to some general results on Brauer groups. We should point out that there is also nothing particularly special about the real anisotropic conic that's used in our arguments, except that it provides us with a particular set of equations to work with. All of our results, and most of our arguments, generalize to arbitrary infinite fields; throughout the article we point out how to make the appropriate generalizations.

To illustrate the algorithm concretely, we conclude the paper with an explicit worked example (Example \ref{Concluding-Ex}), computing the canonical form and multipliers for the tensor square of the Quillen bundle.

\subsection*{Methodology}
The starting point for our work is a direct generalization of the usual clutching construction that's utilized in \cite{MR662762}, for vector bundles on $\mathbb{P}^1$, to the setting of nonsingular plane conics. In the standard clutching construction, one observes that any vector bundle on $\mathbb{P}^1$ is trivialized on the cover \[U_0=\mathbb{P}^1\setminus \{\infty\}\quad \mbox{and}\quad U_\infty=\mathbb{P}^1\setminus \{0\}.\] Here each open $U_0$ and $U_\infty$ is isomorphic with the affine line $\mathbb{A}^1$.

In our framework, one works with a conic $C\subset \mathbb{P}^2$, one chooses lines $l_1\neq l_2\subset \mathbb{P}^2$, and one considers the cover \[U_1=C\setminus (C\cap l_1)\quad \mbox{and}\quad U_2=C\setminus (C\cap l_2).\] For a nonsingular conic $C$, every vector bundle on $C$ is trivialized over this cover. Moreover, if $C$ is anisotropic, then one can choose lines $l_1$ and $l_2$ such that $U_1$ and $U_2$ are a particular type of torus torsor (see Remark \ref{rem: toruscover} for a more precise description). In our main example of a real anisotropic conic, we take $l_1=V(z)$ and $l_2=V(y)$. The opens $U_1$ and $U_2$ are then isomorphic to the unique nontrivial $S^1$-torsor, where $S^1$ is the real unit circle form of $\mathbb{G}_{m,\mathbb{C}}$.

In Section \ref{sec: clutch}, we explain how one can classify line bundles on $C$. In principle, the rational function $l_1/l_2$ defines the transition function for a nontrivial generator of the Picard group of $C$. We then explain, in Section \ref{sec: rank2bundles}, how results on Brauer groups can be used to deduce the decomposition of a vector bundle as stated in Theorem A above. In the process, we explicitly deduce a transition matrix for the indecomposable rank $2$ Quillen bundle $\mathcal{Q}$, see Lemma \ref{lem: 2x2}. These explicit $1\times 1$- and $2\times 2$-transition functions, along with the associated description of vector bundles given in Theorem A, allows one to explicitly describe all possible canonical forms for transition functions on the cover $U_1,U_2$, cf.\ Theorem \ref{thm: cantypes}.

Our algorithm for reducing an arbitrary transition matrix $M$ to its block-diagonal canonical form described above then goes as follows. We first identify a quadratic Galois splitting field $F/k$ of our conic $C$, together with a splitting isomorphism $\mathbb{P}^1_F\cong C_F$. Up to composing with a linear fractional transformation, we can assume that the four points $(l_1\cap C)_F$ and $(l_2\cap C)_F$ are sent to $\{0,1\}$ and $\{\lambda,\infty\}$, respectively, under this isomorphism. Utilizing the Smith normal form for polynomial matrices, we can remove the zeros and poles of $M$ except, possibly, for zeros and poles at $0$ and $\infty$ to get an equivalent matrix $M'$. 

At this point one can use any algorithm (e.g.\ from either of \cite{MR662762, Ephremidze2025}) that finds the Grothendieck-Birkhoff factorization of this new matrix $M'$, and read the equivalent canonical form of $M$ from the resulting factorization (Theorem \ref{thm: alg}). We note that, in order to obtain our polynomial complexity bounds, we can't use a direct adaptation of the Hazewinkel-Martin algorithm from \cite{MR662762}. Unfortunately, the procedure that's described in \cite{MR662762} allows for possible growth in the degrees of the matrix entries, and no polynomial complexity bound is available for this growth. Instead, we make use of the algorithm of Ephremidze \cite{Ephremidze2025}, which avoids this growth issue and allows for a simple polynomial complexity analysis. Finally, if one also requires explicit multipliers achieving this reduction, then they can be found by solving a small set of linear systems over the ground field $k$ (Theorem \ref{thm: multipliers}).

\subsection*{Future Directions}
The explicit algebraic and matrix-theoretic nature of our results was motivated by future applications involving connections. To make this precise, recall that given a monodromy representation $\rho$ of $\mathbb{P}^1_\mathbb{C}\setminus \{p_1,...,p_m\}$, for $m$ points $p_1,...,p_m$, Deligne's canonical extension produces a vector bundle with logarithmic connection $(\mathcal{V}_{Log(\rho)}, \nabla_{Log(\rho)})$ over $\mathbb{P}^1_\mathbb{C}$. Because of the Grothendieck-Birkhoff description of vector bundles on $\mathbb{P}_\mathbb{C}^1$, there is an interest in understanding $\mathcal{V}_{Log(\rho)}$ from the description of $\rho$; we refer the reader to \cite{Yep25, Yepez2025} where this question is considered in the case $m\leq 3$.

One can choose that the points $p_1,...,p_m$ are arranged into conjugate pairs under the action of $\mathrm{Gal}(\mathbb{C}/\mathbb{R})$ via an explicit splitting $\varphi: \mathbb{P}^1_\mathbb{C} \xrightarrow{\sim} C_\mathbb{C}$, so that the configuration of points
is naturally defined in the anisotropic conic $C$ over $\mathbb{R}$ to begin with. This provides an action of the Galois group $\mathrm{Gal}(\mathbb{C}/\mathbb{R})$ on $(\mathcal{V}_{Log(\rho)}, \nabla_{Log(\rho)})$. A natural question, in the spirit of 
\cite{Kat70, Katz90}, is then:
\begin{quote}
\textit{Under what conditions on the monodromy representation $\rho$ does the 
logarithmic connection $(\mathcal{V}_{Log(\rho)}, \nabla_{Log(\rho)})$ admit a 
$\mathrm{Gal}(\mathbb{C}/\mathbb{R})$-equivariant structure, and hence descend 
to a vector bundle with connection defined on the anisotropic 
conic $C$ over $\mathbb{R}$?}
\end{quote}
The canonical transition matrices of Theorem B provide an explicit framework in which the Galois action on the bundle is concretely describable. A descent of the connection would amount to a compatibility condition between the connection matrix of $\nabla_{Log(\rho)}$ and this Galois action, expressible directly in terms of the matrices $L$ and $R$ produced by our algorithm. We leave the systematic investigation of this question to future work.

\vspace{0.3cm}
\noindent\textbf{Acknowledgments}. The second author would like to thank Wayne Raskind for the opportunity to carry out part of this work during his time as an IDA Postdoctoral Fellow, and Jeff Vanderkam for his constant support and hospitality throughout that period. He would also like to thank the Mathematics Department at Princeton University and IDA/CCR-P for providing an excellent research environment, 
and Nick Katz for the many insights that have shaped his mathematical outlook during his time as his postdoctoral advisee.

\section{Clutching and Vector Bundles over Conics}\label{sec: clutch}
Throughout this section we work over the real numbers $\mathbb{R}$, and we use the quadratic extension $\mathbb{C}/\mathbb{R}$ with Galois group $\mathrm{Gal}(\mathbb{C}/\mathbb{R})=\{1,\tau\}$ where $\tau$ is complex conjugation. All of the results in this section can be suitably modified to work over an arbitrary base field $k$, making use of an appropriate separable quadratic extension $F/k$. We point out the necessary modifications in Remark \ref{rem: toruscover} below.

Consider the real anisotropic conic \begin{equation}\label{eq: anicon}\tag{RAC} C=V(x^2+y^2+z^2)\subset \mathbb{P}^2_{x,y,z}.\end{equation} Our goal is to describe a clutching construction that allows us to classify all vector bundles on $C$ in analogy with the Grothendieck-Birkhoff theorem for $\mathbb{P}^1$. In this section, our goal is to identify a cover of $C$ by open subschemes $U_1$ and $U_2$ such that all vector bundles on $C$ are trivialized over this cover.

Set $U_1=C \setminus (C\cap V(z))$ and $U_2= C\setminus (C\cap V(y))$. 
Both $U_1$ and $U_2$ are basic affine open subschemes of $C=\mathrm{Proj}(\mathbb{R}[x,y,z]/(x^2+y^2+z^2))$. One can check that, if \[R_1=\mathbb{R}\left[\frac{x}{z}, \frac{y}{z}\right]\bigg/\left(\left(\frac{x}{z}\right)^2+\left(\frac{y}{z}\right)^2+1\right)\mbox{ and } R_2=\mathbb{R}\left[\frac{x}{y}, \frac{z}{y}\right]\bigg/\left(\left(\frac{x}{y}\right)^2+\left(\frac{z}{y}\right)^2+1\right),\] then $U_i=\mathrm{Spec}(R_i)$ for $i=1,2$. 
Note that $R_i\otimes_\mathbb{R} \mathbb{C}\cong \mathbb{C}[s,1/s]$, e.g.\ for $R_1$ via the map \[\mathbb{C}[s,1/s]\cong R_1\otimes_{\mathbb{R}} \mathbb{C}\quad \mbox{defined by} \quad s\mapsto \frac{x}{z}\otimes i-\frac{y}{z}\otimes 1.\] 
In particular, we have \[R_i=(\mathbb{C}[s,1/s])^{\mathrm{Gal}(\mathbb{C}/\mathbb{R})}\] where $\tau\in \mathrm{Gal}(\mathbb{C}/\mathbb{R})$ acts as $\tau(\lambda)=\overline{\lambda}$ for all $\lambda\in \mathbb{C}$ and $\tau(s)=-1/s$.

Using this action, we can show directly that:

\begin{Lm}
If $M$ is a finitely generated projective module over $R_i$, for either $i=1,2$, then $M$ is free.
\end{Lm}

\begin{proof}
The rings $R_i$ are both Dedekind domains (which can be checked after base change along the extension $\mathbb{C}/\mathbb{R})$. 
As such, all finitely generated projective modules are sums of invertible modules, cf.\ \cite[Corollary 1.4.6]{MR1282290}. 
So, it suffices to show that $\mathrm{Pic}(R_i)=0$.

Let $M$ be an invertible module over $R_i$. Since $M\otimes_{\mathbb{R}} \mathbb{C}$ is free as $\mathbb{C}[s,1/s]$ is a principal ideal domain, we can determine $M$ (up to isomorphism) by the class of a $1$-cocycle from the Galois cohomology group 
\begin{eqnarray}
    \nonumber
    \mathrm{H}^1(\mathrm{Gal}(\mathbb{C}/\mathbb{R}), \mathrm{Aut}_{\mathbb{C}}(\mathbb{C}[s,1/s])).
\end{eqnarray}
If we write $\alpha:\mathrm{Gal}(\mathbb{C}/\mathbb{R})\rightarrow \mathrm{Aut}_{\mathbb{C}}(\mathbb{C}[s,1/s])$ for a normalized $1$-cocycle representative of this class with $\alpha(1)$ the identity, then $\alpha$ is determined completely by the image of $\tau$.

We make the identification $\mathrm{Aut}_{\mathbb{C}}(\mathbb{C}[s,1/s])=\coprod_{i\in \mathbb{Z}} \mathbb{C}^\times s^i$. 
We can then write \[\alpha(\tau)=\lambda s^n\] for some $n\in \mathbb{Z}$. 
The cocycle condition gives us the relation \[1=\alpha(\tau)\cdot \tau(\alpha(\tau))=(\lambda s^n)\cdot ((-1)^n\overline{\lambda}s^{-n})=(-1)^n|\lambda|^2.\] 
This equation can hold for a given $\lambda\in \mathbb{C}$ if and only if $n \equiv 0 \pmod{2}$. 
In this case we also find that $|\lambda|=1$.

We claim that any such $1$-cocycle is trivial. 
To see this, assume first $m=n/2$ is even. 
Write $\lambda=e^{i\theta}$ and set $c=e^{-i(\theta/2)}s^{-m}$. 
Then \[\alpha(1)=1=c^{-1}\cdot c\] and \[\alpha(\tau)=\lambda s^n=e^{i(\theta/2)}s^m\cdot e^{i(\theta/2)}s^m=c^{-1}\cdot \tau(c).\] Hence $\alpha$ is trivial, as claimed. 
If, on the other hand, $m=n/2$ is odd then write $\lambda=e^{i\theta}$ and set $c=-e^{-i(\theta/2+\pi/2)}s^{-m}$. 
Then \[\alpha(1)=1=c^{-1}\cdot c\] and \[\alpha(\tau)=\lambda s^n=-e^{i(\theta/2+\pi/2)}s^m \cdot(-1)^{m+1}(e^{i(\theta/2+\pi/2)})s^m=c^{-1}\cdot \tau(c)\] again.

It follows that \[\mathrm{H}^1(\mathrm{Gal}(\mathbb{C}/\mathbb{R}), \mathrm{Aut}_{\mathbb{C}}(\mathbb{C}[s,1/s]))=0\] and hence $\mathrm{Pic}(R_i)=0$ too.
\end{proof}

Now let $\mathcal{V}$ be an arbitrary vector bundle on $C$. 
Write $\mathcal{V}_i=\mathcal{V}|_{U_i}$ and $\mathcal{V}_{12}=\mathcal{V}|_{U_1\cap U_2}$. 
Since $\mathcal{V}_i$ is free for each of $i=1,2$ by the above lemma, we have that $\mathcal{V}_{12}$ is free as well. 
We want to describe all such $\mathcal{V}$.

One can check that $U_1\cap U_2=\mathrm{Spec}(R_{12})$ where \[R_{12}=\mathbb{R}\left[\frac{x}{y},\frac{x}{z}, \frac{z}{y},\frac{y}{z}\right]\bigg/\left(\left(\frac{x}{y}\right)^2+\left(\frac{z}{y}\right)^2+1\right).\] 
Each vector bundle $\mathcal{V}$ on $C$ is determined on the cover $\mathcal{U}=\{U_1,U_2\}$ by the class of a \v{C}ech $1$-cocycle from the Zariski \v{C}ech cohomology sets $\mathrm{H}^1(\mathcal{U},\mathrm{GL}_n(\mathcal{O}_C))$ varying $n\geq 1$. Any such \v{C}ech $1$-cocycle is determined exactly by one element of $\mathrm{GL}_n(R_{12})$. Two such $1$-cocycles, considered as matrices $M,N\in \mathrm{GL}_n(R_{12})$, define equivalent bundles if and only if there are matrices $V_1\in \mathrm{GL}_n(R_1)$ and $V_2\in \mathrm{GL}_n(R_2)$ such that \[V_2MV_1=N.\]
In the next section, given a vector bundle $\mathcal{V}$ on $C$, we describe a canonical representative $M\in \mathrm{GL}_n(R_{12})$ for this bundle.

\begin{Rem}\label{rem: coverarbfield}
Now let $k$ be an arbitrary field and $C\subset \mathbb{P}^2$ be any smooth conic over $k$. A cover $U_1, U_2$ of $C$ that trivializes any vector bundle $\mathcal{V}$ over $C$ be can be gotten in a similar way to the above, as we now explain.

There are two cases: either $C\cong \mathbb{P}^1$ or $C(k)=\emptyset$. Since the clutching construction for $\mathbb{P}^1$ corresponding to the cover by two opens, each obtained from removing a $k$-rational point, is well-known in the first case, we focus on the second case where $C(k)=\emptyset$.

Let $H_1,H_2\subset \mathbb{P}^2$ be two distinct lines. B{\'e}zout's theorem implies that $\mathrm{deg}(C\cap H_i)=2$. It's well-known that there is an isomorphism $\mathrm{CH}_0(C)\cong \mathbb{Z}$ with a generator the class of any closed point on $C$ with residue field of degree $2$ over $k$, see \cite[Corollary 7.3]{MR2278759}. It follows that both $C\cap H_i$ for $i=1,2$ generate the Chow group $\mathrm{CH}_0(C)$.

For each of $i=1,2$ the localization sequence associated to the closed embedding $C\cap H_i\subset C$ is a short exact sequence \[\mathrm{CH}_0(C\cap H_i)\rightarrow \mathrm{CH}_0(C)\rightarrow \mathrm{CH}_0(C\setminus (C\cap H_i))\rightarrow 0.\] The leftmost map in the sequence above is an isomorphism by the previous paragraph. In particular, if we set $U_i=C\setminus (C\cap H_i)$ for $i=1,2$ then the $U_i$ cover $C$ and $\mathrm{Pic}(U_i)=\mathrm{CH}_0(U_i)=0$. Also, each $U_i\subset \mathbb{A}^2=\mathbb{P}^2\setminus H_i$ is affine with coordinate ring a Dedekind domain. Therefore the vanishing $\mathrm{Pic}(U_i)=0$ again implies that every vector bundle on the opens $U_i$ for $i=1,2$ are trivial.
\end{Rem}

\begin{Rem}\label{rem: toruscover}
In the primary example of this article, the real anisotropic conic $C\subset \mathbb{P}^2$, we use a cover by open subschemes $U_1,U_2$ which satisfy the additional conditions:
\begin{itemize}
\item the complement of each open $U_i$ is a closed point $p_i$ with residue field $\mathbb{C}/\mathbb{R}$,
\item each open subscheme $U_i$ is naturally a torsor for the torus $S^1=\mathrm{Res}_{\mathbb{C}/\mathbb{R}}(\mathbb{G}_{m})/\mathbb{G}_{m}$,
\item under the connecting map of group cohomology $\delta:\mathrm{H}^1(\mathbb{C}/\mathbb{R}, S^1)\rightarrow \mathrm{H}^2(\mathbb{C}/\mathbb{R}, \mathbb{G}_m)$ coming from the short exact sequence \[ 1\rightarrow \mathbb{G}_m\rightarrow \mathrm{Res}_{\mathbb{C}/\mathbb{R}}(\mathbb{G}_m)\rightarrow S^1\rightarrow 1,\] we have $\delta([U_i])=[C]$ for the canonical class $[C]\in \mathrm{Br}(\mathbb{C}/\mathbb{R})=\mathrm{H}^2(\mathbb{C}/\mathbb{R}, \mathbb{G}_m)$.
\end{itemize}
Although these properties aren't necessary for us, they are interesting properties of the cover $U_1,U_2$ that may be useful in other contexts. In this remark, we show how one can construct, for an arbitrary anisotropic conic over any base field $k$, a cover of the form found in Remark \ref{rem: coverarbfield} but, satisfying additional conditions analogous to those above.

So, let $k$ be any field that admits a smooth anisotropic conic $C\subset \mathbb{P}^2$ defined over $k$. Then there exists a separable field extension $F/k$ of degree $2$ (hence a Galois extension of $k$) and an isomorphism $C_F\cong \mathbb{P}^1_F$. We can therefore find two distinct closed points $p_1,p_2$ on $C$ with residue field $F$.

Each point $p_i$ splits over $F$ as a union of two distinct $F$ points, \[p_i\times_k F = \{p_{i,1}, p_{i,2}\}\subset C_F \subset \mathbb{P}^2_F.\] Let $H_i'$ be the unique $F$-line inside $\mathbb{P}^2_F$ containing both $p_{i,1}$ and $p_{i,2}$. The action of the Galois group $\mathrm{Gal}(F/k)$ swaps the points $p_{i,1}$ and $p_{i,2}$, for each of $i=1,2$, and so $\mathrm{H}_i'$ must be a $\mathrm{Gal}(F/k)$-stable subspace of $\mathbb{P}^2_F$. As such, it descends to a line $H_i\subset \mathbb{P}^2$ that satisfies $C\cap H_i=p_i$. The two complements $U_i=C\setminus p_i$ for $i=1,2$ are of the form found in Remark \ref{rem: toruscover} and, as explained in \cite[Section 3]{MR4321613}, we have:

\begin{itemize}
\item each open subscheme $U_i$ is naturally a torsor for the torus $T=\mathrm{Res}_{F/k}(\mathbb{G}_{m})/\mathbb{G}_{m}$,
\item under the connecting map of group cohomology $\delta:\mathrm{H}^1(F/k, T)\rightarrow \mathrm{H}^2(F/k, \mathbb{G}_m)$ coming from the short exact sequence \[ 1\rightarrow \mathbb{G}_m\rightarrow \mathrm{Res}_{F/k}(\mathbb{G}_m)\rightarrow T\rightarrow 1,\] we have $\delta([U_i])=[C]$ for the canonical class $[C]\in \mathrm{Br}(F/k)=\mathrm{H}^2(F/k, \mathbb{G}_m)$.
\end{itemize}
\end{Rem}

\section{Decomposition of Corresponding Matrices}\label{sec: rank2bundles}
In this section we have three goals: 
\begin{enumerate}
\item First, we show that every vector bundle on an anisotropic conic $C\subset \mathbb{P}^2$ -- over an arbitrary base field -- decomposes into a direct sum of indecomposable bundles of rank at most $2$.
\item Second, we describe explicit transition matrices for both the rank $2$ Quillen bundle and for all line bundles on the real anisotropic conic (\ref{eq: anicon}). 
\item Third, we provide another proof of the well-known result that all vector bundles on the real anisotropic conic can be gotten as a sum of line bundles and twists of the Quillen bundle.
\end{enumerate}
The combination of (1), (2), and (3) above allow us to describe a complete list of transition matrices for the cover $U_1,U_2$ from Section \ref{sec: clutch} that describes all vector bundles on the conic (\ref{eq: anicon}).

\begin{Prop}\label{prop: decomp1or2}
Let $k$ be an arbitrary field. Suppose that $C\subset \mathbb{P}^2$ is an anisotropic conic over $k$. Let $\mathcal{V}$ be a nonzero vector bundle on $C$. Then $\mathcal{V}$ decomposes into a direct sum of indecomposable vector bundles of rank at most $2$.

In particular, if $\mathcal{V}$ is an indecomposable vector bundle on $C$, then $\mathrm{rank}(\mathcal{V})\leq 2$.
\end{Prop}

\begin{proof}
Let $F/k$ be a degree $2$ Galois extension splitting $C$, i.e.\ such that there exists an isomorphism $C_F\cong \mathbb{P}^1_F$. Then by the classical Grothendieck-Birkhoff theorem \cite{MR662762}, there is an isomorphism \[\mathcal{V}_F \cong \bigoplus_{m\in \mathbb{Z}} \mathcal{O}(m)^{\oplus n_m}\] for integers $n_m\geq 0$, which we explicitly allow to possibly be $0$ to simplify notation. Write $\mathcal{W}_m$ for the summand $\mathcal{O}(m)^{\oplus n_m}$ appearing in this decomposition.

Let $m$ be an integer such that $n_m\neq 0$, and set $\mathcal{W}=\mathcal{V}_F/\mathcal{W}_m$. The idempotent matrix below, \[\begin{bmatrix} 1_{\mathcal{W}_m} & 0 \\ 0 & 0 \end{bmatrix}\in \begin{bmatrix} \mathrm{End}(\mathcal{W}_m) & \mathcal{W}\otimes\mathcal{W}_m^\vee \\ \mathcal{W}_m\otimes \mathcal{W}^\vee & \mathrm{End}(\mathcal{W})\end{bmatrix}=\mathrm{End}(\mathcal{V}_F).\] is $\mathrm{Gal}(F/k)$-invariant. It therefore descends to an idempotent of $\mathrm{End}(\mathcal{V})$ that provides a decomposition \[\mathcal{V}=\mathcal{V}_m\oplus \mathcal{V}'\] for a vector bundle $\mathcal{V}_m$ such that $(\mathcal{V}_m)_F\cong \mathcal{W}_m$. By induction on the rank, we find such bundles for all $m\in \mathbb{Z}$.

Restricting to the generic point $\eta$ of $C$ yields an injection \[\mathrm{End}(\mathcal{V}_m)\otimes_k k(C)\rightarrow\mathrm{End}((\mathcal{V}_m)_\eta)\cong \mathrm{M}_{n_m}(k(C))\] that must be an isomorphism for dimension reasons. It follows from Amitsur's theorem \cite[Theorem 5.4.1]{MR3727161} that the class of $\mathrm{End}(\mathcal{V}_m)$ in $\mathrm{Br}(k)$ is in the subgroup generated by the class of the conic $C$. Since the class $[C]\in \mathrm{Br}(k)$ has index $2$, it follows that one of the following must be true:
\begin{itemize}
\item $\mathrm{End}(\mathcal{V}_m)$ is either split, i.e.\ $\mathrm{End}(\mathcal{V}_m)\cong M_{n_m}(k)$, 
\item or there is an isomorphism $\mathrm{End}(\mathcal{V}_m)\cong M_{r}(Q)$ for a quaternion algebra $Q$, corresponding to the conic $C$, and with $r=n_m/2$.
\end{itemize}
In the split case, the bundle $\mathcal{V}_m$ decomposes into a direct sum of $n_m$ bundle summands corresponding to the $n_m$ diagonal idempotents of $M_{n_m}(k)$ (i.e.\ those matrices with all but one diagonal entry equal to zero, and the last entry equal to $1$). In the quaternion case, the bundle $\mathcal{V}_m$ decomposes into $r=n_m/2$ summands corresponding to the $r$-many similarly defined diagonal idempotents again. This proves the claim.
\end{proof}

For an arbitrary Severi--Brauer variety $X$, of which a conic is an example, there is an isomorphism $\mathrm{Pic}(X)\cong \mathbb{Z}$ induced from extending scalars to an algebraic closure, cf.\ \cite[Section 2]{MR657430}. In the case of the real anisotropic conic (\ref{eq: anicon}), this isomorphism can be realized explicitly using the defining equations of the conic, which allows one to isolate explicit transition functions for the elements of this group.

In the lemma below we assume that $C\subset \mathbb{P}^2$ is the real anisotropic defined in (\ref{eq: anicon}). We set $\mathcal{U}=\{U_1,U_2\}$ with $U_1=D_+(z)$ and $U_2=D_+(y)$, as in Section \ref{sec: clutch}. 

\begin{Lm}\label{lem: 1x1}
We have $\mathrm{H}^1(\mathcal{U},\mathcal{O}_C^\times)=\mathbb{Z}$ generated by the class of the $1$-cocycle $-z/y$.
\end{Lm}

\begin{proof}
The $\check{\text{C}}$ech cohomology group $\mathrm{H}^1(\mathcal{U},\mathcal{O}_C^\times)$ is computed as the quotient group of the units on the overlap $U_1\cap U_2$ by the product of the units on the open sets $U_1,U_2$:
$$\mathrm{H}^1(\mathcal{U},\mathcal{O}_C^\times) \cong \frac{\mathcal{O}_C(U_1 \cap U_2)^\times}{\mathcal{O}_C(U_1)^\times \cdot \mathcal{O}_C(U_2)^\times} = \frac{R_{12}^\times}{R_1^\times R_2^\times}.$$

We have the following unit groups:
\begin{itemize}
    \item $R_1^\times = R_2^\times = \mathbb{R}^\times$ (the only invertible functions on the opens $U_i$ are constants),
    \item and $R_{12}^\times = \coprod_{i \in \mathbb{Z}} \mathbb{R}^\times \left(\frac{z}{y}\right)^i$.
\end{itemize}
Thus, the product of the unit groups in the denominator is $R_1^\times R_2^\times = \mathbb{R}^\times \cdot \mathbb{R}^\times = \mathbb{R}^\times$.

The quotient group is thus:
$$\mathrm{H}^1(\mathcal{U},\mathcal{O}_C^\times) \cong \frac{\coprod_{i \in \mathbb{Z}} \mathbb{R}^\times \left(\frac{z}{y}\right)^i}{\mathbb{R}^\times} \cong \mathbb{Z}.$$
Here the class of $a(z/y)^n$, for $a\in \mathbb{R}^\times$, is sent to the integer $n$.
\end{proof}

\begin{Rem}\label{rem: can}
Let $\Omega_C$ be the sheaf of K\"ahler differentials on $C$. 
Let $\Omega_1$ and $\Omega_2$ be its restriction to $U_1,U_2$. 
Then $\Omega_1$ is a free module generated by \[\omega_z=(y/z)d(x/z)-(x/z)d(y/z)\] and $\Omega_2$ is a free module generated by \[\omega_y=(z/y)d(x/y)-(x/y)d(z/y).\] 
There are relations:
\[d(x/y)=-(z^3/y^3)d(x/z)\quad\mbox{and}\quad (z/y)d(y/z)=-(y/z)d(z/y)\] on $U_1\cap U_2$. 
One can check that these give \[\omega_y=-(z/y)\omega_z\] so that $\Omega_C$ is also described by the transition function $-y/z$ on $U_{12}$ (from $U_1$ to $U_2$).
\end{Rem}

There is a canonical indecomposable bundle of rank two on any anisotropic conic $C$, over an arbitrary field $k$, called the \textit{Quillen bundle}. To see this, note that \[\dim \mathrm{Ext}^1_C(\mathcal{O}_C,\Omega_C)=\dim \mathrm{H}^1(C,\Omega_C)=1,\] which can be checked over an algebraic closure $\bar{k}$ of $k$. Any nonzero element $\eta$ of $\mathrm{H}^1(C,\Omega_C)\cong \mathrm{Ext}^1_C(\mathcal{O}_C,\Omega_C)$ gives rise to an exact sequence (unique up to equivalence) \begin{equation}\tag{ES}\label{eq: es}0\rightarrow \Omega_C\rightarrow \mathcal{Q}\rightarrow \mathcal{O}_C\rightarrow 0\end{equation} which is, after scalar extension from $k$ to $\bar{k}$, equivalent to the Euler sequence on $\mathbb{P}^1_{\bar{k}}$. Note also that we have $\delta(1)=\eta$ under the connecting map $\delta: \mathrm{H}^0(C,\mathcal{O}_C)\rightarrow \mathrm{H}^1(C,\Omega_C)$ coming from the long exact sequence associated to (\ref{eq: es}), so that $\dim \mathrm{H}^0(C,\mathcal{Q})=0$.

Now if $\mathcal{Q}$ decomposed as a sum $\mathcal{Q}\cong \mathcal{L}\oplus \mathcal{L}'$ of line bundles $\mathcal{L},\mathcal{L}'$ on $C$, then since $\Omega_C\cong \det(\mathcal{Q})\cong \mathcal{L}\otimes \mathcal{L}'$ we must have $\mathcal{L}\cong \Omega_C^{\otimes r}$ and $\mathcal{L}'\cong \Omega_C^{\otimes 1-r}$ for some $r\in \mathbb{Z}$. Counting dimensions, however, one can see that if this were the case then we would necessarily have $\dim \mathrm{H}^0(C,\mathcal{Q})>0$, which contradicts the above paragraph.

The following lemma provides the $1$-cocycle that describes the Quillen bundle.

\begin{Lm}\label{lem: 2x2}
There exists an indecomposable bundle of rank $2$ on the conic $C$ of (\ref{eq: anicon}). This bundle is determined by the $1$-cocycle 
\[M=\begin{bmatrix} -y/z  & x/z  \\ 0  & 1 \end{bmatrix}\in \mathrm{GL}_2(R_{12})\]
considered in $\mathrm{H}^1(\mathcal{U},\mathrm{GL}_2(\mathcal{O}_C))=\mathrm{GL}_2(R_2)\backslash \mathrm{GL}_2(R_{12})/\mathrm{GL}_2(R_1)$.
\end{Lm}

\begin{proof}
By the comments above the statement of Lemma \ref{lem: 2x2}, it suffices to produce a non-split short exact sequence such as (\ref{eq: es}). Define maps \[\phi_2:R_2\xrightarrow{\begin{bmatrix}\frac{x}{y} \\ \frac{z}{y}\end{bmatrix}} R_2\oplus R_2 \quad \mbox{and}\quad \psi_2:R_2\oplus R_2\xrightarrow{\begin{bmatrix}-\frac{z}{y} & \frac{x}{y}\end{bmatrix}} R_2\] and also \[ \phi_1:R_1\xrightarrow{\begin{bmatrix} 0 \\ -1\end{bmatrix}} R_1\oplus R_1 \quad \mbox{and} \quad \psi_1:R_1\oplus R_1\xrightarrow{\begin{bmatrix}1 & 0\end{bmatrix}} R_1.\] Letting $R_i'$ represent the restriction of $R_i$ to $R_{12}$, we have a commutative diagram with exact rows \[
\begin{tikzcd}
0\arrow{r} & R_2'\arrow["\phi_2"]{r} & R_2'\oplus R_2' \arrow["\psi_2"]{r} & R_2' \arrow{r}\arrow[equals]{d} & 0\\
0\arrow{r} & R_1'\arrow["\phi_1"]{r}\arrow["-\frac{y}{z}"]{u} & R_1'\oplus R_1' \arrow["\psi_1"]{r}\arrow["M"]{u} & R_1' \arrow{r} & 0.
\end{tikzcd}\]
Let $V$ be the vector bundle associated to $M$. Then we get an exact sequence \[0\rightarrow \Omega_C\rightarrow V \rightarrow \mathcal{O}_C\rightarrow 0\] from this commutative diagram. We want to show that this sequence is not split.

To do this, we use the connecting morphism \[\mathcal{O}_C(C)\xrightarrow{\delta} \mathrm{H}^1(\mathcal{U},\Omega_C).\] Note that we have $\psi_2(z/y,-x/y)=1$ and $\psi_1(1,0)=1$. Since $-x^2/(yz)=z/y+y/z$, we have that $\phi_2(-x/z)=(z/y+y/z,-x/y)$. It follows that we have that $\delta(1)=-(x/z)\omega_y$ in $\mathrm{H}^1(\mathcal{U},\Omega_C)$ considered as a quotient of $\Omega_C(U_{12})$. 

We want to show that the $1$-cocycle $\delta(1) = -(x/z)\omega_y$ is non-trivial in $\mathrm{H}^1(\mathcal{U},\Omega_C)$. Assume, for a contradiction, that $\delta(1)$ is trivial or, equivalently, that $-(x/z)\omega_y$ is a $1$-coboundary; i.e.\ assume that there exist sections $h_1 \in \Omega_C(U_1)$ and $h_2 \in \Omega_C(U_2)$ such that $-(x/z)\omega_z = h_2 - h_1$ on $U_{12}$.

Since $\Omega_C(U_1) = R_1 \omega_z$ and $\Omega_C(U_2) = R_2 \omega_y$, we have $h_1=f_1\omega_z$ and $h_2=f_2\omega_y$, where $f_1 \in R_1$ and $f_2 \in R_2$. Using the transition relation $\omega_y=-(z/y)\omega_z$ (from Remark \ref{rem: can}), the coboundary condition becomes:
$$-\frac{x}{z}\omega_y = f_2\omega_y - f_1\left(-\frac{y}{z}\omega_y\right) \quad \text{on } U_{12}.$$
Equating the coefficients in $R_{12}$ yields the required decomposition:
\begin{equation}\label{eq: f1} -\frac{x}{z} = f_2 + f_1\left(\frac{y}{z}\right), \quad \text{with } f_1 \in R_1, f_2 \in R_2.\end{equation}

Let $P_y = C \cap V(z)$ and $P_z = C \cap V(y)$ so that $R_1$ consists of rational functions regular outside of $P_y$ and $R_2$ consists of functions regular outside of $P_z$. Note that $P_y\cap P_z=\emptyset$.
We can analyze the regularity of the function $f_1 = -\frac{x}{y} - f_2\left(\frac{z}{y}\right)$ on $P_y$ (e.g.\ by using the valuation on the function field $\mathbb{R}(C)$ determined by $z$).
\begin{itemize}
    \item The function $\frac{x}{y}$ is regular on $P_y$ (since $y \neq 0$ when $z=0$ on the conic).
    \item The function $v=\frac{z}{y}$ has a simple zero at $P_y$. Since $f_2 \in R_2$, it is regular at $P_y$. Thus, the product $f_2(z/y)$ is also regular at $P_y$.
\end{itemize}
Since both terms are regular at $P_y$, we find that $f_1$ must be regular at $P_y$. Since $f_1 \in R_1$ is regular everywhere except possibly at $P_y$, and we showed it is also regular at $P_y$, it follows that $f_1$ must be a regular function on the entire projective curve $C$. But, since $C$ is a smooth projective curve over $\mathbb{R}$, we have $\mathcal{O}_C(C) = \mathbb{R}$, so $f_1 = c \in \mathbb{R}$.

Substituting $f_1=c$ back into the decomposition (\ref{eq: f1}), we find for $f_2$:
$$f_2 = -\frac{x}{z} - c\frac{y}{z}.$$
Since $f_2$ must be in $R_2$, it must be regular everywhere except possibly at $P_z$. However, the function $f_2 = -\frac{x}{z} - c\frac{y}{z}$ has a pole at $P_y$ (where $z=0$). Since $P_y \cap P_z=\emptyset$, we find $f_2$ is not regular at $P_y$. This contradicts the requirement that $f_2 \in R_2$.
Therefore, $\delta(1)\neq 0$, and the sequence is non-split. It follows that the bundle $V$ constructed as above is isomorphic to the Quillen bundle $\mathcal{Q}$ on the conic $C$.
\end{proof}

Finally, we have:

\begin{Thm}\label{thm: cantypes}
Let $C=V(x^2+y^2+z^2)\subset \mathbb{P}^2$ be the real anisotropic conic. Let $\mathcal{V}$ be an arbitrary vector bundle on $C$. Then $\mathcal{V}$ is a direct sum of bundles of the form \[\Omega_C^{\otimes r} \quad \mbox{and}\quad \mathcal{Q}\otimes\Omega_C^{\otimes s}\] where $\mathcal{Q}$ is the rank 2 indecomposable Quillen bundle given in Lemma \ref{lem: 2x2}.

In particular, any vector bundle $\mathcal{V}$ is described by a canonical transition matrix $M(\mathcal{V})$ on $U_1\cap U_2$ as a block sum $M(\mathcal{V})=\bigoplus_{i\in \mathbb{Z}} M_i^{\oplus m_i}$, for uniquely determined integers $m_i\geq 0$, where \[M_i= (-y/z)^{i/2} \quad \mbox{or}\quad M_i=\begin{bmatrix} (-y/z)^{(i+1)/2}  & (x/z)(-y/z)^{(i-1)/2}  \\ 0  & (-y/z)^{(i-1)/2} \end{bmatrix}\] depending if $i\equiv 0 \pmod{2}$ or $i\equiv 1 \pmod{2}$ respectively.
\end{Thm}

\begin{proof}
By Proposition \ref{prop: decomp1or2}, we know that $\mathcal{V}$ decomposes as a sum line bundles and vector bundles of rank $2$. Since every line bundle is a twist of $\Omega_C$ by Lemma \ref{lem: 1x1} and Remark \ref{rem: can}, it suffices to classify all rank $2$ bundles on $C$.

So let $\mathcal{F}$ be an arbitrary rank $2$ bundle on $C$. Base changing to $\mathbb{C}/\mathbb{R}$, we get a bundle $\mathcal{F}_{\mathbb{C}}$ on $\mathbb{P}^1_{\mathbb{C}}$ which must decompose as a sum $\mathcal{F}_{\mathbb{C}}\cong \mathcal{O}(m)\oplus \mathcal{O}(n)$ for integers $n,m\in \mathbb{Z}$. We note that if $m\neq n$, then there are canonical idempotents of $\mathrm{End}(\mathcal{F}_{\mathbb{C}})$ which descend to idempotents of $\mathrm{End}(\mathcal{F})$, giving a decomposition of $\mathcal{F}$ into line bundles on $C$.

So assume $\mathcal{F}_\mathbb{C} \cong \mathcal{O}(n)\oplus \mathcal{O}(n)$. Since $(\Omega_C)_{\mathbb{C}}\cong \mathcal{O}(-2)$, we can twist $\mathcal{F}$ by $\Omega_C^{\otimes s}$ for some $s\in \mathbb{Z}$ to assume that either $n=0$ or $n=-1$. Now let $\mathcal{E}$ be either the sum $\mathcal{O}_C\oplus \mathcal{O}_C$, if $n=0$, or the rank two bundle from Lemma \ref{lem: 2x2}, if $n=-1$. Note that, if $n=0$, then $\mathrm{End}(\mathcal{E})\cong M_2(\mathbb{R})$ and, if $n=-1$, then $\mathrm{End}(\mathcal{E})\cong Q$ is isomorphic to the real quaternions (cf.\ the proof of Proposition \ref{prop: decomp1or2} and the claim of Lemma \ref{lem: 2x2}).

In either case, $\mathrm{Hom}(\mathcal{F},\mathcal{E})$ is a left $\mathrm{End}(\mathcal{E})$-module of real dimension $4$, which can be checked after extension along $\mathbb{C}/\mathbb{R}$. There is a Zariski open subset of the affine space $\mathrm{Hom}(\mathcal{F},\mathcal{E})$ consisting of those maps which are isomorphisms; moreover, this subset is nonempty since it has a point over $\mathbb{C}$. But, over an infinite field like $\mathbb{R}$, every nonempty open subset of an affine space has a rational point.
\end{proof}

\section{An Algorithm for Finding Canonical Forms}
In this section we outline an algorithm for determining the block-diagonal canonical form, as specified in Theorem \ref{thm: cantypes}, of an arbitrary invertible matrix $M\in\mathrm{GL}_n(R_{12})$ with $n\geq 1$, where as before we write \[R_{12}=\mathbb{R}\left[\frac{x}{y},\frac{x}{z}, \frac{z}{y},\frac{y}{z}\right]\bigg/\left(\left(\frac{x}{y}\right)^2+\left(\frac{z}{y}\right)^2+1\right).\] Our algorithm finds the canonical form of $M$ \textit{without} also finding the ``multipliers", i.e.\ without also determining matrices $L\in \mathrm{GL}_n(R_2)$ and $R\in \mathrm{GL}_n(R_1)$ such that $LMR$ is the associated canonical form. We show how, however, that if the multipliers $L$ and $R$ are needed, then one can efficiently find them given $M$ and its canonical form.

\begin{Rem}
Although we work only with real coefficients throughout this section, our algorithms will work replacing the base field by any infinite field $k$, if suitably adjusted. In this case the quadric $C=V(x^2+y^2+z^2)\subset \mathbb{P}^2_k$ may or may not be anisotropic. If $C$ is anisotropic, e.g.\ if $k$ is a totally real number field, then the statement and classification of canonical forms as given in Theorem \ref{thm: cantypes} also still holds, and our algorithm below will apply with the appropriate modifications.
\end{Rem}

We use $u,v$ for the standard coordinate functions of $\mathbb{P}^1_\mathbb{C}$ and $x,y,z$ for the standard coordinate functions of $\mathbb{P}^2_\mathbb{R}$. Over $\mathbb{C}$, the morphism \[\phi:\mathbb{P}^1_\mathbb{C}\rightarrow \mathbb{P}^2_\mathbb{C}\quad \quad [u:v]\mapsto [i(u^2-2uv+2v^2): 2v(u-v): u(2v-u)]\] defines an isomorphism between $\mathbb{P}^1_{\mathbb{C}}$ and $C_{\mathbb{C}}$, the extension of the conic of (\ref{eq: anicon}) to $\mathbb{C}$. It follows from these formulae that \[\phi(\mathbb{P}^1_{\mathbb{C}})\cap V(z) = \phi(V(u(2v-u)))\quad \mbox{and} \quad \phi(\mathbb{P}^1_{\mathbb{C}})\cap V(y) = \phi(V(2v(u-v))).\] In particular, using the rational coordinate $w=u/v$ on $\mathbb{P}^1_{\mathbb{C}}$, the open subset $U_{1,\mathbb{C}}$ from Section \ref{sec: clutch} is identified with $\phi(\mathbb{P}^1_\mathbb{C}\setminus \{0, 2\})$ and $U_{2,\mathbb{C}}$ is identified with $\phi(\mathbb{P}^1_\mathbb{C}\setminus \{\infty, 1\})$.

\begin{Lm}\label{lem: refine}
Let $\mathcal{W}=\{W_1, W_2\}$ be the open cover of $\mathbb{P}^1_{\mathbb{C}}$ where $W_1=\mathbb{P}^1_{\mathbb{C}}\setminus \{0\}$ and $W_2=\mathbb{P}^1_{\mathbb{C}}\setminus\{\infty\}$. Then for any $n\geq 1$ the map \[\mathrm{H}^1(\mathcal{W},\mathrm{GL}_n(\mathcal{O}_{\mathbb{P}^1_{\mathbb{C}}}))\rightarrow \mathrm{H}^1(\mathcal{U},\mathrm{GL}_n(\mathcal{O}_{\mathbb{P}^1_{\mathbb{C}}})),\] induced by the canonical refinement $\mathcal{U}=\{U_{1,\mathbb{C}}, U_{2,\mathbb{C}}\}$ of $\mathcal{W}$, is a bijection.
\end{Lm}

\begin{proof}
The map is clearly injective. Since both sets are in canonical bijection with the set of rank $n$ vector bundles on $\mathbb{P}^1_{\mathbb{C}}$, the map is also surjective.
\end{proof}

Using the formulae \begin{equation}\label{eq: phi}\frac{x}{y} = \frac{i(w^2-2w+2)}{2(w-1)}\quad \mbox{and}\quad -\frac{z}{y} = \frac{w^2-2w}{2(w-1)}\end{equation} on $C_{\mathbb{C}}$, any matrix in $\mathrm{GL}_n(R_{12}\otimes_{\mathbb{R}}\mathbb{C})$ can be identified with a matrix in $\mathrm{GL}_n(S)$ where $S=\mathbb{C}[w,w^{-1},(w-1)^{-1}, (w-2)^{-1}]$ is naturally the coordinate ring of $\phi^{-1}(U_{1,\mathbb{C}}\cap U_{2,\mathbb{C}})$. By Lemma \ref{lem: refine}, any matrix from $\mathrm{GL}_n(S)$ can be transformed by a sequence of left-right multiplications to a matrix in $\mathrm{GL}_n(\mathbb{C}[w,w^{-1}])$ preserving the corresponding vector bundle isomorphism class on $\mathbb{P}^1_{\mathbb{C}}$. It's possible to do this constructively, as we now show.

First, in order to analyze the complexity of our algorithms, we'll need the following:

\begin{Defn}
Let $M\in M_n(\mathrm{GL}_n(\mathbb{C}(w))$ be a square $n\times n$-matrix with rational entries in the variable $w$. We define \[ d_w(M):=\max_{1\leq i,j \leq n} \{|\deg_w(M_{i,j})|\}\] as the largest (in magnitude) degree in $w$ that appears in an entry of $M$. Here we write $\deg_w(p(w))=\max\{\deg_w(f(w)),\deg_w(g(w))\}$ for a rational function $p(w)=\frac{f(w)}{g(w)}$ of two polynomials $f(w),g(w)\in \mathbb{C}[w]$.
\end{Defn}

\begin{Lm}\label{lem: reduce}
Let $M\in \mathrm{GL}_n(S)$ be an invertible $n\times n$-matrix with coefficients in the ring $S=\mathbb{C}[w,w^{-1},(w-1)^{-1},(w-2)^{-1}]\cong R_{12}\otimes_\mathbb{R} \mathbb{C}$. Then there exists an algorithm -- that requires at most a polynomial number of field operations in the inputs $n$ and $d_w(M)$ -- for finding matrices \[U_1\in \mathrm{GL}_n(\mathbb{C}[w^{-1},(w-2)^{-1}])\quad \mbox{and}\quad U_2\in\mathrm{GL}_n(\mathbb{C}[w,(w-1)^{-1}])\] such that $U_2MU_1\in \mathrm{GL}_n(\mathbb{C}[w,w^{-1}])$. More precisely, the algorithm detailed here has an algebraic complexity on the order of $O(\mathrm{Poly}(n,d_w(M)))$. 
\end{Lm}

\begin{proof}
Let $M$ be a matrix from $\mathrm{GL}_n(\mathbb{C}[w,w^{-1}, (w-1)^{-1}, (w-2)^{-1}])$. Set \[r_0=\min_{1\leq i,j \leq n}\{\nu_0(M_{i,j}), 0\},\quad r_1=\min_{1\leq i,j \leq n}\{\nu_1(M_{i,j}), 0\}, \quad \mbox{and} \quad r_2 = \min_{1\leq i,j\leq n}\{\nu_2(M_{i,j}),0\}\] where $\nu_t(M_{i,j})$ is the valuation of the $(i,j)$-entry of $M$ at the point $t\in \mathbb{P}^1_\mathbb{C}$. By extracting denominators, we can write \begin{equation}\label{eq: poles} M=w^{r_0}(w-1)^{r_1}(w-2)^{r_2}M'\end{equation} for a matrix $M'$ with coefficients in $\mathbb{C}[w]$ and $d_w(M')\leq 2d_w(M)$. Note that since the determinant $\det(M)$ is a unit in $\mathbb{C}[w,w^{-1},(w-1)^{-1},(w-2)^{-1}]$, and as $M'$ has entries in $\mathbb{C}[w]$, it follows that the polynomial $\det(M')$ has zeros only at the points $w=0,1,2$.

Since $\mathbb{C}[w]$ is a principal ideal domain, we can write $M'$ in Smith normal form \[M'=LDR,\quad\mbox{where } L,R\in \mathrm{GL}_n(\mathbb{C}[w])\] for some diagonal matrix $D$ with coefficients in $\mathbb{C}[w]$. Since $L$ and $R$ have constant determinant, it follows that the entries of $D$ have nonzero valuation only at $w=0,1,2$. 
For our complexity analysis, we point out that finding the normal form $D$, along with left and right multipliers $L$ and $R$, can be accomplished in a polynomial number of field operations in $n$ and $d_w(M)$ by \cite[Corollary 4.1]{MR1378100}. 

Note, by \cite[proof of Lemma 4.1]{MR1378100}, one can find $L$ with $d_w(L^{-1})\leq (n-1)d_w(M')$. Note also that both $L$ and $L^{-1}$ are contained in $\mathrm{GL}_n(\mathbb{C}[w])\subset \mathrm{GL}_n(\mathbb{C}[w,(w-1)^{-1}])$ which agrees with $\mathrm{GL}_n(R_2\otimes_\mathbb{R} \mathbb{C})$ under the above identifications.

For each $1\leq i \leq n$, let $k_i=\nu_1(D_{i,i})$. Define $E_1$ as the diagonal matrix with $i$th diagonal entry $(w-1)^{-r_1-k_i}$ and note that $E_1\in \mathrm{GL}_n(R_2\otimes_\mathbb{R} \mathbb{C})$ as well. Multiplying the equation (\ref{eq: poles}) by $E_1L^{-1}$ on the left gives \[E_1L^{-1}M=w^{r_0}(w-1)^{r_1}(w-2)^{r_2}E_1DR=w^{r_0}(w-2)^{r_2}D'R\] where $D'$ is a diagonal matrix with coefficients in $\mathbb{C}[w]$ with zeros only at $w=0,2$. 

Now set $N=E_1L^{-1}M$, which is a matrix with coefficients in $\mathbb{C}[w,w^{-1}, (w-2)^{-1}]$ and set $s_0=\min_{1\leq i,j \leq n}\{\nu_\infty(N_{i,j}),0\}$ where $\nu_\infty(N_{i,j})$ is valuation of $N_{i,j}$ at $w=\infty$. Since $w-2=w\left(1-2/w\right)$ we can similarly write \begin{equation}\label{eq: poles2} N = (1/w)^{s_0+r_2}(w-2)^{r_2}N'= w^{-s_0}(1-2/w)^{r_2}N'\end{equation} for a matrix $N'$ with coefficients in $\mathbb{C}[w^{-1}]$. Note that the determinant $\det(N)$ may have poles or zeros only at $w=0,2$ so that $\det(N')$ may have poles or zeros only at $w=0,2$ as well. For our complexity analysis, we also point out that $d_w(E_1)\leq d_w(M)+2nd_w(M)$, since each $k_i$ is an invariant factor of $M$, so that \[d_w(N')\leq 2d_w(N)\leq 2(d_w(E_1)+d_w(L^{-1})+d_w(M))\leq 8nd_w(M)\] is bounded by a polynomial in $d_w(M)$.

Since $\mathbb{C}[w^{-1}]$ is a principal ideal domain, we can write $N'$ in Smith normal form \[N'=UQV, \quad \mbox{where } U,V\in \mathrm{GL}_n(\mathbb{C}[w^{-1}])\] and where $Q$ is a diagonal matrix with coefficients in $\mathbb{C}[w^{-1}]$. Since $U,V$ have constant determinant, it follows that the entries of $Q$ have nonzero valuation only at $w=0,2$. Using \cite[Corollary 4.1]{MR1378100} again, we can find $Q$ along with $U$ and $V$ using at most a polynomial number of operations in $n$ and $d_w(N')\leq 4nd_w(M)$. We also point out that, by \cite[proof of Corollary 4.1]{MR1378100}, one can find $V$ with $d_w(V^{-1})\leq nd_w(N')\leq 8n^2d_w(M)$. Note too that $V$ and $V^{-1}$ are contained in $\mathrm{GL}_n(\mathbb{C}[w^{-1}])\subset \mathrm{GL}_n(\mathbb{C}[w^{-1},(w-2)^{-1}])$ which agrees with $\mathrm{GL}_n(R_1\otimes_\mathbb{R} \mathbb{C})$.

For each $1\leq i \leq n$, let $l_i=\nu_2(Q_{i,i})$. Define $E_2$ as the diagonal matrix with $i$th diagonal entry $(1-2/w)^{-r_2-l_i}$ and note that $E_2\in \mathrm{GL}_n(R_1\otimes_\mathbb{R}\mathbb{C})$ since we have \[\left(1-\frac{2}{w}\right)\cdot \left(1+\frac{2}{w-2}\right) = 1,\] as one can easily check. Multiplying (\ref{eq: poles2}) on the right yields \[NV^{-1}E_2=w^{-s_0}(1-2/w)^{r_2}UQE_2 = w^{-s_0}UQ'\] for a diagonal matrix $Q'$ with coefficients in $\mathbb{C}[w^{-1}]$ and with poles only at $w=0$. Using similar reasoning, all of this can be done in a polynomial number of operations in $n$ and $d_w(M)$ in the base field. For later, we note that since each $l_i$ is an invariant factor of $N'$, we have $d_w(E_2)\leq d_w(M)+nd_w(N')\leq (8n^2+1)d_w(M)$.

In total, we've found $E_1L^{-1}\in \mathrm{GL}_n(R_2\otimes_\mathbb{R}\mathbb{C})$ and $V^{-1}E_2\in \mathrm{GL}_n(R_1\otimes_\mathbb{R}\mathbb{C})$ such that $(E_1L^{-1})M(V^{-1}E_2)$ is an invertible matrix in $\mathrm{GL}_n(\mathbb{C}[w,w^{-1}])$. 
\end{proof}

The above proof also provides an upper bound for the degrees of entries of the multipliers $U_1$ and $U_2$: 

\begin{Cor}\label{cor: bounds}
By following the process outlined by Lemma \ref{lem: reduce}, one can find $U_1$ and $U_2$ such that $U_2MU_1\in \mathrm{GL}_n(\mathbb{C}[w,w^{-1}])$ satisfying \[d_w(U_1)\leq (16n^2+1)d_w(M)\quad\mbox{and}\quad d_w(U_2)\leq(4n-1)d_w(M).\]
\end{Cor}

\begin{proof}
Following the proof of the previous lemma, the matrices $E_1$, $L^{-1}$, $V^{-1}$, and $E_2$ can be chosen so that \[d_w(E_1L^{-1})\leq (4n-1)d_w(M)\quad \mbox{and}\quad  d_w(V^{-1}E_2)\leq (16n^2+1)d_w(M).\] Take $U_2=E_1L^{-1}$ and $U_1=V^{-1}E_2$ constructed in this way.
\end{proof}

We're almost ready to state our main algorithms but, before doing so, we make one last point of notation. For any matrix $M\in \mathrm{GL}_n(R_{12})$, we write $\phi^{-1}(M)$ for the matrix in $\mathrm{GL}_n(\mathbb{C}(w))$ using the formula of (\ref{eq: phi}) above. 

\begin{Thm}\label{thm: alg}
Let $M$ be a matrix from $\mathrm{GL}_n(R_{12})$ where \[R_{12}=\mathbb{R}\left[\frac{x}{y},\frac{x}{z}, \frac{z}{y},\frac{y}{z}\right]\bigg/\left(\left(\frac{x}{y}\right)^2+\left(\frac{z}{y}\right)^2+1\right)\] and for some $n\geq 1$. Let \[R_1=\mathbb{R}\left[\frac{x}{z}, \frac{y}{z}\right]\bigg/\left(\left(\frac{x}{z}\right)^2+\left(\frac{y}{z}\right)^2+1\right)\mbox{ and } R_2=\mathbb{R}\left[\frac{x}{y}, \frac{z}{y}\right]\bigg/\left(\left(\frac{x}{y}\right)^2+\left(\frac{z}{y}\right)^2+1\right).\] 

Then there is an efficient algorithm, i.e.\ an algorithm that is polynomial, in the number of field operations needed, in the inputs: 
\begin{itemize}
\item the size $n$ of the matrix $M$,
\item and the maximal degree $d=d_w(\phi^{-1}(M))$ of an entry in $\phi^{-1}(M)$,
\end{itemize} for finding a block diagonal matrix $D$, of the form specified by Theorem \ref{thm: cantypes}, such that there exist matrices $L\in \mathrm{GL}_n(R_2)$ and $R\in \mathrm{GL}_n(R_1)$ with $LMR=D$.
\end{Thm}

\begin{proof}
Considering $M$ as a matrix after extension of scalars to $R_{12}\otimes_{\mathbb{R}}\mathbb{C}$, we can -- by the paragraphs between Lemma \ref{lem: refine} and Theorem \ref{thm: alg} -- find matrices $U,V$ such that $UMV$ is contained in $\mathrm{GL}_n(\mathbb{C}[w,w^{-1}])$. By using the algorithm of \cite{Ephremidze2025}, for example, we can then find $U'\in \mathrm{GL}_n(\mathbb{C}[w])$ and $V'\in \mathrm{GL}_n(\mathbb{C}[w^{-1}])$ such that $(U'U)M(VV') = D'$ for a diagonal matrix $D'$ with entries $D'_{i,i}=w^{k_i}$ for some $k_i\in \mathbb{Z}$ and for $1\leq i \leq n$. Ordering the integers $k_i$ in ascending order, there is a unique block diagonal matrix $D$ in Theorem \ref{thm: cantypes} that corresponds to this ordered sequence.

The complexity of this algorithm depends on the complexity of the algorithm used to compute the Smith normal form of a polynomial matrix of size $n\times n$, as well as the complexity of the algorithm used in computing a Grothendieck-Birkhoff factorization. It's known that the Smith normal form $M=UNV$ of a matrix $M\in M_n(F[x])$, i.e.\ an $n\times n$-matrix that has polynomial entries with coefficients in a field $F$, can be computed using at most a polynomial (in both $n$ and the maximal degree $d$ of an entry of $M$) number of operations in the base field $F$ asymptotically. A small subtly here is whether or not one requires an algorithm that only finds the shape of $N$, or if one requires also knowledge of the multipliers $U$ and $V$. It seems that \cite{MR1378100} was the first to show that we can find all of $N$, $U$, and $V$ with complexity polynomial (in $n$ and $d$) in the number of operations of the field $F$ asymptotically (see also the more recent \cite{MR2736355} whose introduction gives an overview of the problem of computing Smith normal forms).

As for computing the Grothendieck-Birkhoff factorization, the algorithm provided by \cite{Ephremidze2025} depends on, in each step of a recursive procedure, computing a polynomial (in the integer inputs $n$ and $d$) number of algebraic operations in the base field considered. However, this algorithm may also require a number of steps proportional to the power $m$ appearing in the determinant \[cw^m=\mathrm{det}(U\phi^{-1}(M)Vw^k)\] where $k=d_w(U\phi^{-1}(M)V)$ and where we've written $c$ for the scalar coefficient $c\in \mathbb{C}^\times$. By Corollary \ref{cor: bounds}, however, we know that $k$ is bounded by a polynomial in $n$ and $d$ and therefore so is $m\leq 2nk$. A careful reading of the algorithm of \cite{Ephremidze2025} shows that this completely describes the complexity of the algorithm, as requiring a polynomial number of algebraic field operations in the inputs $n$ and $d$, completing the proof.
\end{proof}

\begin{Thm}\label{thm: multipliers}
Let $M$ be a matrix from $\mathrm{GL}_n(R_{12})$ as before. Suppose also that a block diagonal matrix $D$ is known, of the form specified by Theorem \ref{thm: cantypes}, such that there exist matrices $L\in \mathrm{GL}_n(R_2)$ and $R\in \mathrm{GL}_n(R_1)$ with $LMR=D$.

Then there is an efficient algorithm, i.e.\ an algorithm that is polynomial, in the number of field operations needed, in the inputs: 
\begin{itemize}
\item the size $n$ of the matrix $M$,
\item and the maximal degree $d=d_w(\phi^{-1}(M))$ of an entry in $\phi^{-1}(M)$,
\end{itemize} for finding the matrices $L\in \mathrm{GL}_n(R_2)$ and $R\in \mathrm{GL}_n(R_1)$ such that $LMR=D$.
\end{Thm}

\begin{proof}
With $M$ and $D$ known, we can find $L,R^{-1}$ such that $LM=DR^{-1}$ by solving a system of linear equations. More precisely, the relations \[\left(\frac{x}{y}\right)^2+\left(\frac{z}{y}\right)^2=-1,\quad \left(\frac{x}{z}\right)^2+\left(\frac{y}{z}\right)^2=-1,\quad \frac{x^2}{yz}=-\frac{y}{z}-\frac{z}{y}\] can be used to write \[M=M_0(y/z, z/y) + M_1(y/z, z/y)\frac{x}{y}\] for matrices of Laurent polynomials $M_0$ and $M_1$ in $y/z$. Similarly one can write $D$ in terms of matrices of Laurent polynomials $D_0,D_1$.

We're then reduced to looking for matrices $L=L_0(z/y)+L_1(z/y)\frac{x}{y}$, for matrices of polynomials $L_0$ and $L_1$ in the variable $z/y$, and $R^{-1}=R_0(y/z)+\frac{y}{z}R_1(y/z)\frac{x}{y}$, for matrices of polynomials $R_0$ and $R_1$ in the variable $y/z$, such that $LM=DR^{-1}$, i.e.\ \begin{multline}\label{eq: sys}\left(L_0(z/y)+L_1(z/y)\frac{x}{y}\right)\left(M_0(y/z, z/y) + M_1(y/z, z/y)\frac{x}{y}\right) =\\
\left(D_0(y/z, z/y) + D_1(y/z, z/y)\frac{x}{y}\right)\left(R_0(y/z)+\frac{y}{z}R_1(y/z)\frac{x}{y}\right).
\end{multline}
Since the coefficients in $M_0$, $M_1$, $D_0$, and $D_1$ are fixed, we can compare the coefficients of $(y/z)^i$ and $(y/z)^i(x/y)$, varying over all $i\in \mathbb{Z}$, to get a number of systems of equations that can be solved to find the coefficients of $L_0$, $L_1$, $R_0$, and $R_1$. Note also that if any solution $L$ or $R^{-1}$ is invertible, then, since we are working over an infinite field, almost every solution is invertible (and $\det(L)=\det(R^{-1})\in \mathbb{R}^\times$).

We claim that there exist $L_0$, $L_1$, $R_0$, and $R_1$ solving (\ref{eq: sys}) with \[\deg_{z/y}(L_0),\,\deg_{z/y}(L_1),\,\deg_{y/z}(R_0),\,\deg_{y/z}(R_1) \leq f(n,d),\] for a polynomial $f(n,d)$, such that $L$ and $R^{-1}$ are both invertible. If this were true, then we would be able to find a random solution of (\ref{eq: sys}) by solving at most $2f(n,d)+2$ many linear systems over $\mathbb{R}$ of size $2n\times 2n$, which can be done in a polynomial number of field operations in $n$ and $d$, thereby completing the proof.

By Corollary \ref{cor: bounds}, we get a polynomial upper bound (in the inputs $n$ and $d$) on the degrees of entries of multipliers $U,V$ such that $U\phi^{-1}(M)V$ has coefficients in $\mathbb{C}[w,w^{-1}]$. The algorithm in \cite{Ephremidze2025} gives an upper bound on $d_w(U')$ and $d_w(V')$, for the multipliers $U',V'$ such that $(U'U)M(VV')=D'$ is diagonal with monomial entries, by \[d_w(U'),d_w(V')\leq d_w(UMV)\leq d_w(U)+d_w(M)+d_w(V).\] Hence, after moving to the complex numbers and transforming the equations (\ref{eq: sys}) to the variable $w$, there are solutions to these equations with $L_0,L_1,R_0,R_1$ having degrees bounded by a polynomial $f(n,d)$. But a real linear system with a complex solution also must have a real solution, and a random choice of real solutions will be invertible.
\end{proof}

\begin{Rem}
An algorithm for computing the Grothendieck-Birkhoff factorization is also provided in \cite{MR662762}. Hazewinkel and Martin's algorithm, by comparison to \cite{Ephremidze2025}, requires computing the greatest common divisor of several polynomials and performing row operations to clear entries of the matrix using this divisor. In the process of doing so, however, there is a risk that entries of the matrix elsewhere may grow in size at a rate that can not be bounded by a polynomial number of field operations in $n$ and $d$. Note that both algorithms work for an arbitrary field of coefficients $k$.

Note also that there are a number of typos in \cite{Ephremidze2025}, e.g.\ below equation (15), $S_{n}^+(t)A$ should be $S_n^+(0)A$, the example in equation (17) has some errors on the second line, and the inequality below (18) should be $m_k\leq (Nr+m)-k/2$.
\end{Rem}

\begin{Ex}\label{Concluding-Ex}
Consider the matrix 
\[M = \begin{bmatrix} 
\frac{y^2}{z^2} & -\frac{xy}{z^2} & -\frac{xy}{z^2} & -\frac{y^2}{z^2}-1 \\ 
0 & -\frac{y}{z} & 0 & \frac{x}{z} \\ 
0 & 0 & -\frac{y}{z} & \frac{x}{z} \\ 
0 & 0 & 0 & 1 
\end{bmatrix}\] 
which is the tensor square of the matrix defining the Quillen bundle from Lemma \ref{lem: 2x2}. As such, the canonical form for this matrix is 
\[N = \begin{bmatrix} 
-\frac{y}{z} & 0 & 0 & 0 \\ 
0 & -\frac{y}{z} & 0 & 0 \\ 
0 & 0 & -\frac{y}{z} & 0 \\ 
0 & 0 & 0 & -\frac{y}{z} 
\end{bmatrix},\] 
which can be seen easily over an algebraic closure of the ground field. Solving the equations from $LM=NR^{-1}$ gives the following multipliers:

\[
L = \begin{bmatrix} 
\frac{x}{y} & \frac{z}{y} & \frac{z}{y} & -\frac{x}{y} \\ 
\frac{z}{y} & -\frac{x}{y} & -\frac{x}{y} & -\frac{z}{y} \\ 
-1 & 1 & -1 & -1 \\ 
1 & 1 & -1 & 1 
\end{bmatrix}\quad \mbox{and}\quad R^{-1} = \begin{bmatrix} 
-\frac{x}{z} & -\frac{y}{z} & -\frac{y}{z} & \frac{x}{z} \\ 
-1 & 0 & 0 & -1 \\ 
\frac{y}{z} & -\frac{x}{z} + 1 & -\frac{x}{z} - 1 & -\frac{y}{z} \\ 
-\frac{y}{z} & \frac{x}{z} + 1 & \frac{x}{z} - 1 & \frac{y}{z} 
\end{bmatrix}.\]
\end{Ex}

\bibliographystyle{alpha}
\bibliography{References}

\end{document}